\documentclass[11pt]{amsart}
\usepackage{amsfonts,latexsym,amsthm,amssymb,graphicx}
\usepackage{mathabx}
\usepackage[all]{xy}
\usepackage[usenames]{color} 
\usepackage{tikz}
\usetikzlibrary{shapes}
\usepackage{xypic}
\usetikzlibrary{decorations.pathreplacing}

\DeclareFontFamily{OT1}{rsfs}{}
\DeclareFontShape{OT1}{rsfs}{n}{it}{<-> rsfs10}{}
\DeclareMathAlphabet{\mathscr}{OT1}{rsfs}{n}{it}

\CompileMatrices

\newtheorem{theorem}{Theorem}[section]
\newtheorem{lemma}[theorem]{Lemma}
\newtheorem{corol}[theorem]{Corollary}
\newtheorem{prop}[theorem]{Proposition}

\newtheorem{conj}{Conjecture}

{\theoremstyle{definition} }
{\theoremstyle{remark} \newtheorem{remark}[theorem]{Remark}
\newtheorem{example}[theorem]{Example}}

\newcommand{\caL}{{\mathcal L}}
\newcommand{\cT}{{\mathcal T}}

\newcommand{\h}{\mathfrak{h}}
\newcommand{\g}{\mathfrak{g}}
\DeclareMathOperator{\Gr}{Gr}
\DeclareMathOperator{\SL}{SL}
\newcommand{\C}{{\mathbb C}}
\newcommand{\Z}{{\mathbb Z}}
\newcommand{\cE}{{\mathcal E}}
\newcommand{\cF}{{\mathcal F}}
\newcommand{\cO}{{\mathcal O}}
\newcommand{\cQ}{{\mathcal Q}}
\newcommand{\id}{\text{id}}
\newcommand{\Abb}{{\mathbb{A}}}
\newcommand{\Nbb}{{\mathbb{N}}}
\newcommand{\Pbb}{{\mathbb{P}}}

\newcommand{\one}{1\hskip-3.5pt1}
\newcommand{\csm}{{c_{\text{SM}}}}
\DeclareMathOperator{\Fl}{Fl}

\newcommand{\qede}{\hfill$\lrcorner$}

\title[CSM classes for Schubert cells]{
Chern-Schwartz-MacPherson classes for Schubert cells in flag manifolds
}
\author{Paolo Aluffi}
\author{Leonardo C.~Mihalcea}
\address{
Mathematics Department, 
Florida State University,
Tallahassee FL 32306
}
\email{aluffi@math.fsu.edu}
\address{
Department of Mathematics, 
Virginia Tech University, 
Blacksburg, VA 24061
}
\email{lmihalce@vt.edu}

\begin{document}

\thanks{P.~A. was supported in part by a Simons Collaboration Grant 
and NSA Award H98230-15-1-0027; L.~C.~M. was supported in part
by a Simons Collaboration Grant and
by NSA Young Investigator Award H98230-13-1-0208.}

\begin{abstract}
We obtain an algorithm computing the Chern-Schwartz-MacPherson (CSM) 
classes of Schubert cells in a generalized flag manifold $G/B$. In analogy to 
how the ordinary divided difference operators act on Schubert classes, each 
CSM class of a Schubert class is obtained by applying certain Demazure-Lusztig 
type operators to the CSM class of a cell of dimension one less. These operators
define a representation of the Weyl group on the homology of $G/B$.
By functoriality, we deduce algorithmic expressions for CSM classes of Schubert 
cells in any flag manifold $G/P$. We conjecture that the CSM classes of Schubert 
cells are an effective combination of (homology) Schubert classes, and prove that 
this is the case in several classes of examples. We also extend our 
results and conjectures to the torus equivariant setting.
\end{abstract}

\maketitle

%%%

\section{Introduction}\label{intro}
A classical problem in Algebraic Geometry is to define characteristic 
classes of singular algebraic varieties generalizing the notion of the 
total Chern class of the tangent bundle of a non-singular variety. 
The existence of a functorial theory of Chern classes for possibly
singular varieties was conjectured by Grothendieck and Deligne,
and established by R.~MacPherson \cite{macpherson:chern}.
This theory associates a class $c_*(\varphi)\in H_*(X)$ with every
constructible function $\varphi$ on $X$, such that $c_*(\one_X)
=c(TX)\cap [X]$ if $X$ is a smooth compact complex variety.
(The theory was later extended to arbitrary algebraically closed
fields of characteristic~$0$, with values in the Chow group
\cite{MR1063344}, \cite{MR2282409}.)
The strong functoriality properties satisfied by these classes 
determine them uniquely; we refer to \S\ref{s:cam} below for details.
If $X$ is a compact complex variety, then the class $c_*(\one_X)$ 
coincides with a class defined earlier by M.~H.~Schwartz 
\cite{schwartz:1,schwartz:2} independently of the work mentioned
above. This class is commonly known as the {\em
Chern-Schwartz-MacPherson\/} (CSM) class of $X$. In general, 
we denote by $\csm(W)$ the class $c_*(\one_W) \in H_*(X)$
for any constructible (e.g., locally closed) subset $W\subseteq X$. 
 
Let $G$ be a complex simple Lie group and let $B$ be a Borel subgroup. 
Denote by $W$ the Weyl group of $G$. The goal of this note is to 
provide an algorithm calculating the CSM classes of the Schubert 
cells $X(w)^\circ := B w B / B$ in the generalized flag manifold 
$G/B$, as $w$ varies in  $W$. To describe the answer, we need 
to recall two well known families of operators on the homology group 
$H_*(G/B)$; we refer to \S\ref{s:prelim} below for details. 

Let $s_1, \dots , s_r \in W$ be the simple reflections corresponding 
respectively to the simple roots $\alpha_1, \dots , \alpha_r$ of $G$, 
and let $\ell: W \to \mathbb{N}$ be the length function. Denote by 
$X(w) := \overline{Bw B/B}$ the Schubert variety corresponding to 
$w$; it is a subvariety of $G/B$ of complex dimension $\ell(w)$. 
For each $1 \le k \le r$, the classical BGG operator \cite{BGG} is 
an operator $\partial_k:H_*(G/B) \to H_{*+2}(G/B)$ which sends the 
Schubert class $[X(w)]$ to $[X(w s_k)]$ if $\ell(ws_k) > \ell(w)$ 
and to $0$ otherwise. The Weyl group admits a right action on 
$H_*(G/B)$, which was originally used to {\em define} the BGG 
operator $\partial_k$. For $1 \le k \le r$ define the non-homogeneous 
operator 
\[ 
\cT_k:= \partial_k - s_k : H_*(G/B) \to H_*(G/B) \/,
\] 
where $s_k$ denotes the (right) action of the simple reflection $s_k$. 
The main result of this note is the following.

\begin{theorem}\label{T:mainintro} 
Let $w \in W$ be a Weyl group element and $1 \le k \le r$. Then 
\[ 
\cT_k( \csm(X(w)^\circ)) = \csm(X (w s_k)^\circ) \/. 
\] 
\end{theorem} 
In the case $w=\id$, the Schubert cell $X(\id)^\circ$ is a point, and 
$\csm([pt]) = [pt]$. More generally, if $w= s_{i_1} \ldots s_{i_k}$, 
then the theorem implies that the CSM class $\csm(X(w)^\circ)$ 
is obtained by composing the operators $\cT_{i_k} \cdots \cT_{i_1}$. 
This is reminiscent of the classical situation in Schubert Calculus, 
where one generates all the Schubert classes by applying 
successively the BGG operators $\partial_k$. To further the analogy, 
$\partial_k^2 = 0$ and the BGG operators satisfy the braid relations. 
We prove that $\cT_k^2=1$, and then Theorem~\ref{T:mainintro} can 
be used to show that the operators $\cT_k$ also satisfy the braid relations 
(Proposition~\ref{prop:T}). In particular, there is a well defined operator 
$\cT_w$ associated with any $w \in W$, and this yields a new representation 
of the Weyl group on $H_*(G/B)$. The CSM classes of the Schubert cells 
$X(w)^\circ$ are the values obtained by applying to the class of a point $[pt]$ 
the operators $T_{w^{-1}}$ in this representation.
Further, the action of each simple reflection $\cT_k$ on 
Schubert classes can be written explicitly using the Chevalley 
formula (Proposition~\ref{prop:Tk}). This gives an {\em explicit\/} 
algorithm to calculate the CSM class of any Schubert cell.

We also note that $\cT_k$ is related to a specialization of the 
{\em Demazure-Lusztig operator\/} defined in \cite[\S12]{ginzburg}, 
in relation to degenerate Hecke algebras. 
We plan on investigating this connection further 
in a future paper.

Perhaps the most surprising feature of the CSM classes (and of the 
operators $\cT_k$) is a positivity property. It follows from the definition 
of CSM classes that 
\[ 
\csm(X(w)^\circ) = \sum_{v \le w} c(v;w) [X(v)] \/, 
\] 
where $\le$ denotes the Bruhat ordering, and $c(w;w) = c(\id;w) = 1$. 
Despite the fact that $\cT_k$ does not preserve the positivity of a combination
of Schubert classes, we conjecture that $c(v;w) > 0$ for all $v \le w$. 
We have checked this by explicit computations for thousands of Schubert 
cells $X(w)$ in type A flag manifolds $\Fl(n)$ for $n \le 8$.
We were also able to prove that this positivity holds for several families of 
Schubert cells, across all Lie types; see \S\ref{s:posCSM} below.

Let $P \subset G$ be a parabolic subgroup containing $B$, and let 
$p : G/B \to G/P$ be the natural projection. Then $p(X(w)^\circ) = 
X(w W_P)^\circ$, where $W_P$ is the subgroup of $W$ generated 
by the reflections in $P$. The functoriality of CSM classes can be 
used to prove that 
\[ 
p_* ( \csm(X(w)^\circ)) = \csm(X(w W_P)^\circ)
\] 
(see Proposition~\ref{prop:pfcsm} below). In particular, the (conjectured) 
positivity of CSM classes of Schubert cells in $G/B$ implies the positivity 
of classes in any $G/P$. In the case when $G/P$ is a Grassmann 
manifold this was proved in several cases by the authors of this note
\cite{aluffi.mihalcea:csm, mihalcea:binomial}, B.~Jones \cite{MR2628830}, 
J.~Stryker \cite{stryker}, and it was settled for all cases by J.~Huh \cite{huh:csm}. 
Huh was able to realize the homogeneous components of CSM classes 
as a class of an effective cycle, but unfortunately his main technical 
requirements do not seem to hold for arbitrary flag manifolds $G/B$. 

Our calculation of CSM classes is based on a construction of these 
classes in terms of bundles of logarithmic tangent fields. For every
$W\subseteq X$, the class $\csm(\one_W)\in H_*(X)$ may be obtained by
pushing forward to $X$ the total Chern class of $T_{\widetilde W}(-\log D)$,
where $\widetilde W$ is a resolution of $\overline W$ such that 
$D:=\widetilde W\smallsetminus W$ is a simple normal crossing (SNC)
divisor (\cite{aluffi:diff}, \cite{MR2209219}). We refer to \S\ref{s:cam} for details.
This approach was used successfully in previous work by the authors
\cite{aluffi.mihalcea:csm} to compute CSM classes of Schubert cells 
in the Grassmann manifold and prove partial positivity results. In that case one can use a resolution 
of a Schubert variety which is a (smooth) Schubert variety, but in a 
partial flag manifold. B.~Jones \cite{MR2628830} gave an alternative computation of the
classes, by means of a different (small) resolution, and also obtained
partial positivity results.
The resolution used in~\cite{aluffi.mihalcea:csm} has finitely many orbits 
of the Borel subgroup $B$, and this was a key fact in Huh's full proof of
the positivity conjecture for CSM classes in that case \cite{huh:csm}. 

For generalized flag manifolds $G/B$ a resolution is given by 
{\em Bott-Samelson varieties}. These are iterated $\mathbb{P}^1$-bundles, 
and they can be constructed from any (possibly non-reduced) word 
consisting of simple reflections. Section~\ref{ss:BS} is devoted to the 
definition and cohomological properties of Bott-Samelson varieties. 
The key technical result in the paper is Theorem~\ref{thm:recursion}, 
establishing the necessary cohomological formulas calculating the 
push-forward of the Chern class of the logarithmic tangent bundle. 
The operators $\partial_k$ and $s_k$, and $\cT_k = \partial_k - s
_k$, appear naturally in these 
push-forward formulas.
Properties of the operators $\cT_k$ are discussed 
in \S\ref{s:propTk}, and in~\S\ref{s:posCSM} we formulate the positivity
conjecture and a related identity and provide partial evidence for these
statements. In~\S\ref{s:equiv} we consider the torus-equivariant setting,
and prove (Theorem~\ref{thm:equimain}) that a direct analogue of 
Theorem~\ref{T:mainintro} holds for the {\em equivariant\/} CSM classes 
of Schubert cells. We also propose equivariant generalizations of the 
conjectures presented in~\S\ref{s:posCSM}.

Equivariant Chern-Schwartz-MacPherson classes were defined 
by T.~Ohmoto \cite{ohmoto:eqcsm} and have recently been the object of further
study: A.~Weber \cite{MR2928940} proved localization formulas for these
classes, and in \cite{RV} R.~Rim\'anyi and A.~Varchenko use Weber's results 
to prove that equivariant CSM classes of Schubert cells agree with the $\kappa$ 
classes they studied in earlier work \cite{MR3229982}. 
The $\kappa$ classes are related to the stable envelopes of D. Maulik and A. Okounkov 
\cite{maulik.okounkov:quantum}, which are classes in the equivariant cohomology 
of the cotangent bundle of flag manifolds. In future work we plan to investigate 
further the connection between the CSM classes and stable envelopes, both from 
the localization point of view (cf.~\cite{RV} and \cite{su:restriction}) and from the 
point of view of Ginzburg's theory of bivariant Chern classes \cite{ginzburg:characteristic}.
\smallskip

{\em Acknowledgements.} We are grateful to Mark Shimozono for pointing 
out some algebraic identities between the operators $\partial_k$ and 
$s_k$; and to Dave Anderson, J\"org Sch\"urmann, and Chanjiang Su for 
useful discussions. 

%%%

\section{Preliminaries}\label{s:prelim}
The goal of this section is to recall some basic facts on the
cohomology of flag manifolds and on Bott-Samelson resolutions. We
refer to \cite{brion.kumar:frobenius}, \S2.1--\S 2.2, for more details and 
references to the standard literature.

\subsection{Flag manifolds and Schubert varieties}\label{ss:FmaSv}
Let $G$ be a complex simple Lie group and 
let $T \subseteq B \subseteq G$
be a maximal torus included in a Borel subgroup of $G$. Let $\h$ and
$\g$ be the Lie algebras of $T$ and $G$, and let $R \subseteq \h^*$ be
the associated root system with the set of positive roots $R^+$
determined by $B$. Denote by $\Delta:= \{\alpha_1, \dots, \alpha_r \}
\subseteq R^+$ the set of simple roots. Let $R^\vee $ denote the set of
coroots $\alpha^\vee \in \h$ and $\langle \cdot , \cdot \rangle: R
\otimes R^\vee \to \Z$ the evaluation pairing.

To each root $\alpha \in R$ one associates 
a reflection $s_\alpha$ in the
Weyl group $W=N_G(T)/T$. The set of simple reflections $s_i:=
s_{\alpha_i}$ generate $W$, thus each $w \in W$ can be written as $w=
s_{i_1} \cdots s_{i_k}$. The minimal such integer $k$ is denoted by
$\ell(w)$, the length of $w$. If $k = \ell(w)$, then the decomposition
$w= s_{i_1} \cdots s_{i_k}$ is said to be
{\em reduced}. 
There is a
partial order on $W$ called the {\em Bruhat ordering}, defined as
follows: $u<v$ if there exists a chain 
\[ 
u_0:=u \to u_1:= u s_{\beta_1} \to u_2:=u_1 s_{\beta_2} \to \cdots
\to u_n:=v=u_{n-1} s_{\beta_n} 
\] 
where the
$\beta_i$'s are roots in $R$ such that $\ell(u_i) > \ell(u_{i-1})$.
Let $G/B$ be the generalized flag variety. This is a homogeneous
space for $G$ (hence a non-singular variety) and it is stratified by
{\em Schubert cells} $X(w)^\circ:= B w B/ B$, where $w \in W$, and 
each such cell $X(w)^\circ$ is isomorphic to $\C^{\ell(w)}$. The closure
$X(w):= \overline{BwB/B}$ is called a {\em Schubert variety}. Each
Schubert variety $X(w)$ has a fundamental class $[X(w)] \in H_{2
  \ell(w)}(G/B)$, and these classes form a $\Z$-basis for the
(co)homology.
It may be verified that $X(v)\subseteq X(w)$ if and only if $v\le w$ in
the Bruhat order. In fact, $X(w)=\bigcup_{v\le w} X(v)^\circ$. It follows
(\cite[Example 1.9.1]{fulton:it}) that any class in $H_*(X(w))$ may
be written uniquely as an integer linear combination
$\sum_{v\le w} c_v [X(v)]$.

Let $\h^*_\Z$ be the integral weight lattice and let
$\lambda \in \h^*_\Z$ be an integral weight.~Then one
constructs the $G$-equivariant line bundle over $G/B$ 
\[
\caL_\lambda:= G \times^B \C_{-\lambda} = (G \times \C)/B 
\]
where $B$ acts on $G \times \C$ by $b. (g,u) = (g b^{-1},
\lambda(b)^{-1} u)$ (and the action of $B= UT$ on $\C$ extends the
action of $T$ so that it is trivial over the unipotent group $U$).

The {\em Chevalley formula\/} states that
\begin{equation}\label{E:Chevalley}
c_1(\caL_\lambda)\cdot [X(w)] = \sum \langle \lambda,\beta^\vee\rangle
[X(ws_\beta)]\/,
\end{equation}
where the sum is over all positive roots $\beta$ such that $\ell(ws_\beta)
=\ell(w)-1$. See e.g.,~\cite[Lemma 8.1]{fulton.woodward:quantum}.

\subsection{Two operators acting on $H^*(G/B)$}\label{ss:twoops} 
For each simple root $\alpha_k \in \Delta$ one can construct the BGG 
operator $\partial_k:H^*(G/B) \to H^*(G/B)$ defined in \cite{BGG} as follows. Let
$P_k \subseteq G$ be the minimal parabolic subgroup corresponding to
$\alpha_k$. Then the natural projection $\pi:G/B \to G/P_{k}$ is a
$\Pbb^1$-bundle and there is a fiber square 
\[ 
\xymatrix{
G/B \times_{G/P_k} G/B \ar[rr]^-{pr_1}\ar[d]^{pr_2} && G/B \ar[d]^{p_k} \\ 
G/B \ar[rr]^{p_k}&& G/P_k} 
\] 
The BGG operator is defined to be $\partial_k = p_k^* (p_k)_*$. 
We record next two properties of this operator; see e.g., 
\cite[Proposition 2]{knutson:noncomplex} or 
\cite{tymoczko:permaction} for simple proofs.

\begin{prop}\label{prop:BGGop} The operator $\partial_k$ satisfies 
the following properties:
\begin{itemize}
\item[(a)] For all Weyl group elements $w \in W$,
\begin{equation}\label{E:partial} 
\partial_k([X(w)])= 
\begin{cases} 
0 & \textrm{if } \ell(ws_k) < \ell(w) \\  
[X(ws_k)] & \textrm{if } \ell(ws_k) > \ell(w)
\end{cases} 
\end{equation} 
In particular, $\partial_k^2 = 0$ and the BGG operators satisfy the
same braid relations as the elements of the Weyl group.
\item[(b)] For all $\gamma_1, \gamma_2 \in H^*(G/B)$,
\[ 
\partial_k(\gamma_1 \gamma_2) = \partial_k(\gamma_1) \gamma_2 
+ \gamma_1 \partial_k(\gamma_2) - c_1(\caL_{\alpha_k}) 
\partial_k(\gamma_1) \partial_k(\gamma_2) 
\] where $c_1(\caL_{\alpha_k})$ denotes the Chern class of
$\caL_{\alpha_k}$.
\end{itemize}
\end{prop}

For each $w \in W$ there is a well defined map $r_w:G/T \to G/T$
obtained by multiplying on the right with any lift of $w$ in
$N_G(T)$. This induces a ring endomorphism $r_w^*: H^*(G/T) \to
H^*(G/T)$. Note that the projection $G/T \to G/B$ is a $U \simeq
B/T$-bundle and because $U$ is contractible this implies that the
cohomology rings $H^*(G/B)$ and $H^*(G/T)$ are isomorphic. This
defines a right action of $W$ on $H^*(G/B)$, denoted again by $w$;
it will be clear from the context whether we refer to the Weyl group
element or to its action on $H^*(G/B)$. It is well known (see e.g., 
\cite[\S 1]{BGG}) that for $w=s_k$ this operator satisfies
\begin{equation}\label{E:sk} 
s_k =  \id - c_1(\caL_{\alpha_k}) \partial_k 
\end{equation} 
where the Chern class $c_1(\caL_{\alpha_k})$ acts on $H^*(G/B)$ by
multiplication. Combining this with the Chevalley formula
we obtain an identity 
\begin{equation}\label{E:raction} 
s_k[X(w)] = \begin{cases}
\hphantom{-}    [X(w)] & \textrm { if } \ell(ws_k)< \ell(w); \\ 
    -[X(w)] - \sum \langle \alpha_k, \beta^\vee \rangle [X(ws_k s_\beta)] 
    & \textrm{ if } \ell(w s_k) > \ell(w) . 
\end{cases} 
\end{equation} 
where the sum is over all positive roots $\beta\ne \alpha_k$ such that 
$\ell(w) = \ell(ws_k s_\beta)$.

For use in~\S\ref{s:propTk} we also record the following commutation 
relation of the operators $\partial_k$ and~$s_k$.

\begin{lemma}\label{lem:commdelksk}
With notation as above,
\[
s_k\partial_k = \partial_k\quad,\quad
\partial_k s_k = -\partial_k\/.
\]
In particular, $\partial_k s_k+s_k\partial_k = 0$.
\end{lemma}

\begin{proof}
The first equality follows immediately from the definition of $s_k$
and the fact that $\partial_k^2=0$. The second equality is a
consequence of Proposition~\ref{prop:BGGop}(b):
\begin{align*}
\partial_k s_k & = \partial_k - \partial_k(c_1(\caL_{\alpha_k}) 
\partial_k) \\
& = \partial_k - \partial_k(c_1(\caL_{\alpha_k})) \partial_k
-c_1(\caL_{\alpha_k}) \partial_k^2 + c_1(\caL_{\alpha_k})
\partial_k(c_1(\caL_{\alpha_k})) \partial_k^2 \\
& = \partial_k - \partial_k(c_1(\caL_{\alpha_k})) \partial_k \\
& = -\partial_k\/,
\end{align*}
where we used the fact that $\partial_k(c_1(\caL_{\alpha_k}))
= \langle\alpha_k, \alpha_k^\vee \rangle = 2$
as may be checked using Proposition~\ref{prop:BGGop}(a) and
the Chevalley formula.
\end{proof}

\subsection{Bott-Samelson varieties}\label{ss:BS}
For each word $s_{i_1}s_{i_2}\dots
s_{i_k}$ for an element $w \in W$ one
can construct a tower of $\Pbb^1$ bundles, the {\em Bott-Samelson
variety\/} $Z:= Z_{i_1, \dots, i_k}$, endowed with a map $\theta:= 
\theta_{i_1, \dots, i_k}: Z \to X(w)$. If the word is reduced, this map is
birational, giving a resolution of singularities for $X(w)$ (depending
on the choice of the word for $w$). 
There are several ways to do this, but for our purpose we present 
an inductive construction which can be found e.g., 
in~\cite[\S 2.2]{brion.kumar:frobenius}.

If the word is empty, then define $Z:= pt$. In general assume we have 
constructed $Z':= Z_{i_1, \dots, i_{k-1}}$ and the map $\theta':
Z' \to X(w')$, for $w'= s_{i_1} \cdots s_{i_{k-1}}$. Define 
$Z=Z_{i_1,\dots, i_k}$ so that the left square in the diagram
\begin{equation}\label{E:BSconst} 
\xymatrix@C=50pt{
Z \ar[r]^-{\theta_1} \ar[d]_{\pi} & G/B \times_{G/P_{i_k}}
G/B \ar[r]^-{pr_1}\ar[d]^{pr_2} & G/B \ar[d]^{p_{i_k}} \\ 
Z'\ar[r]^{\theta'} & G/B \ar[r]^{p_{i_k}} & G/P_{i_k}
} 
\end{equation} 
is a fiber square; the morphisms $pr_1, pr_2, p_{i_k}$ are the natural
projections. From this construction it follows that $Z_{i_1,\dots, i_k}$
is a smooth, projective variety of dimension $k$.

The Bott-Samelson variety $Z$ is equipped with a simple normal
crossing (SNC) divisor~$D_Z$. We recall next an explicit inductive
construction of this divisor, which will be needed later. If $Z=pt$, 
then $D_Z =\emptyset$. In general, $G/B$ is the projectivization 
$\Pbb(E)$ of a homogeneous rank-$2$ vector bundle $E\to G/P_k$, 
defined up to tensoring with a line bundle. Define $\cE:=E \otimes
\cO_E(1)$, a vector bundle over $G/B = \Pbb(E)$. Then we
have the Euler sequence of the projective bundle $\Pbb(E)$
\begin{equation}\label{E:seqZ}
\xymatrix{
0 \ar[r] & \cO_{\Pbb(E)} \ar[r] & \cE \ar[r] & \cQ \ar[r] & 0
}
\end{equation}
where $\cQ$ is the relative tangent bundle $T_{p_{i_k}}$.
Note that $\cE$ is independent of the specific choice of $E$, and 
$pr_2: G/B \times_{G/P_{i_k}} G/B \to G/B$, that is, the pull-back 
of $\Pbb(E)$ via $p_{i_k}$, may be identified with~$\Pbb(\cE)$. Let
$\cE':= (\theta')^*\cE$ and $\cQ':=(\theta')^*\cQ$, and pull-back the 
previous sequence via $\theta'$ to get an exact sequence 
\begin{equation}\label{E:seqZ'} 
\xymatrix{
0 \ar[r] & \cO_{Z'} \ar[r] &  \cE' \ar[r] & \cQ' \ar[r] &  0 \,. 
}
\end{equation} 
The inclusion $\cO_{Z'} \hookrightarrow \cE'$ gives a section
$\sigma:Z' \to Z$ of $\pi$ and therefore a divisor $D_k:=\sigma(Z') =
\Pbb(\cO_{Z'})$ in $\Pbb(\cE') =Z$. The SNC divisor on $Z$ is defined by
\[ 
D_Z = \pi^{-1}(D_{Z'}) \cup D_k 
\] 
where $D_{Z'}$ is the inductively constructed SNC divisor on $Z'$.
The following result is well known, see 
e.g., \cite[\S2.2]{brion.kumar:frobenius}:

\begin{prop}\label{prop:resolution} 
If $w$ is a reduced word, then the image of the composition 
$\theta= pr_1 \circ \theta_1: Z_{i_1,\dots,i_k} \to G/B$ is the Schubert 
variety $X(w)$. Moreover, $\theta^{-1}(X(w) \smallsetminus X(w)^\circ) =
D_{Z_{i_1,\dots,i_k}}$ and the restriction map
\[ 
\theta: Z_{i_1, \dots,i_k} \smallsetminus D_{Z_{i_1,\ldots,i_k}} \to X(w)^\circ 
\] 
is an isomorphism. 
\end{prop}

Let $h_k:= c_1(\cO_{\cE'}(1)) \in H^2(Z)$. 
For later use we record next the class of the divisor~$D_k$ in~$Z$,
and the Chern classes of the relative tangent bundle 
$T_\pi= T_{\Pbb(\cE')|Z'}$ and of $T_Z$.

\begin{prop}\label{prop:divclass} 
The following identities hold in $H^*(Z)$:
\begin{itemize}
\item[(a)] $D_k = c_1(\pi^*(\cQ') \otimes \cO_{\cE'}(1)) \in H^*(Z)$;
\item[(b)] $h_k \cdot D_k = 0$;
\item[(c)] $c(T_\pi) = (1 + D_k) (1+h_k)$, and therefore 
\[ 
c(T_Z) = \pi^*(c(T_{Z'})) (1+h_k)(1+D_k) \,.
\]
\end{itemize}
\end{prop}

\begin{proof} 
(a) follows from the definition of $D_k$ and~\cite[Ex.~3.2.17]{fulton:it},
since $\cQ'=\cE'/\cO_{Z'}$.

(b) holds since $h_k|_{D_k}=c_1(\cO_{\cO_{Z'}}(1))=0$. 

To prove (c), note that by \eqref{E:seqZ'} the Chern roots of $\cE'$ 
are $0$ and $c_1(\cQ')$; it follows from~(a) that the Chern roots 
of $\pi^* {\cE'} \otimes \cO_{\cE'}(1)$ are $h_k$ and $D_k$. The Euler sequence
\[ 
\xymatrix{
0 \ar[r] & \cO_Z \ar[r] & \pi^*\cE' \otimes \cO_{\cE'}(1) \ar[r] & T_\pi
\ar[r] & 0
}
\] 
then implies $c(T_\pi) = c( \pi^*\cE' \otimes \cO_{\cE'}(1))=(1+h_k) (1+D_k)$.
The last statement follows from $c(T_Z)= \pi^* (c(T_{Z'}))c(T_\pi) $.
\end{proof}

%%%

\section{Chern-Schwartz-MacPherson classes of Schubert cells in $G/B$} \label{s:cam}

\subsection{CSM classes}\label{ss:cam} 
Let $Y$ be an algebraic variety over $\C$. Denote by $\cF(Y)$ the group of 
constructible functions on $Y$: the elements of $\cF(Y)$ are finite sums 
$\sum c_i \one_{W_i}$ where $c_i \in \Z$, $W_i \subseteq Y$ are locally closed 
subvarieties, and $\one_W$ denotes the characteristic function taking value $1$ 
on $p \in W$ and $0$ otherwise. If $f: Y \to X$ is a proper morphism of varieties, 
one can define a push-forward $f_*:\cF(Y) \to \cF(X)$ by setting
\[ 
f_*(\one_W)(p) = \chi (f^{-1}(p) \cap W)
\] 
for $W \subseteq Y$ a subvariety and $p\in X$, and extending by linearity 
to every $\varphi \in\cF(Y)$; this makes $\cF$ into a covariant functor. 
Here $\chi$ denotes the topological Euler characteristic.
MacPherson \cite{macpherson:chern} proved a conjecture of
Deligne and Grothendieck stating that there exists a
natural transformation $c_*: \cF \to H_*$ such that if $Y$ is
non-singular, then $c_*(\one_Y) = c(T_Y) \cap [Y]$.
The naturality of $c_*$ means that if $f:Y \to X$ is a 
proper morphism, then the following diagram commutes:
\[ 
\xymatrix@C=40pt{ 
\cF(Y) \ar[r]^{c_*} \ar[d]^{f_*} & H_*(Y) \ar[d]^{f_*} \\ 
\cF(X) \ar[r]^{c_*} & H_*(X) 
}
\] 
That is,
\begin{equation}\label{E:funCSM}
f_* (c_*(\varphi)) = c_*(f_*(\varphi))
\end{equation}
in $H_*(X)$, for all constructible functions $\varphi$.
Resolution of singularities and the normalization requirement easily
imply that $c_*$ is unique. 

If $Y$ is a compact complex variety, the class $c_*(\one_Y)$ coincides 
with a class defined earlier by M.~H.~Schwartz
\cite{schwartz:1,schwartz:2}; this class is the {\em
Chern-Schwartz-MacPherson\/} (CSM) class of $Y$. 
Taking $f$ to be a constant map, the commutativity of the above diagram
implies that $\int c_*(\one_{Y})=\chi(Y)$, so this class provides a natural
generalization of the Poincar\'e-Hopf theorem to possibly singular varieties.
Abusing language a little, we denote by $\csm(W) := c_*(1_W)\in H_*Y$ 
the CSM class of any constructible set $W$ in a variety $Y$; by additivity
of Euler characteristics, $\int \csm(W)=\chi(W)$.

Our main tool will be the observation that if $Z$ is a nonsingular variety 
and $W\subseteq Z$ is an open subvariety such that $Z\smallsetminus W$
is a SNC divisor with components~$D_i$, then
\begin{equation}\label{E:csmdefSNC}
\csm(W) = \frac{ c(T_Z)}{ \prod_i (1 + D_i)} \cap [Z]\in H_*Z
\end{equation}
(cf.~\cite[Proposition 15.3]{MR1893006}, \cite[Theorem~1]{aluffi:diff}).

In fact this observation may be used to extend the scope of the natural
transformation $c_*$ to arbitrary algebraically closed fields of characteristic~$0$, 
with values in the Chow group~$A_*$.
In this generality, $c_*$ may be constructed as follows. Every constructible 
function on $Y$ can be written as a linear combination of characteristic functions 
$\one_W$ for $W$ locally closed and non-singular in $Y$, so it suffices to 
describe $\csm(W)=c_*(\one_W)$ for such $W$. By resolution of singularities, 
there exists a desingularization $\pi: Z \to \overline{W}$ of the closure $\overline{W}$ 
of $W$ in $Y$ such that $D:=\pi^{-1}(\overline{W} \smallsetminus W)$ 
is a SNC divisor in $Z$. Then
one may take the push-forward of~\eqref{E:csmdefSNC} to $Y$ as the 
{\em definition\/} of $c_*(\one_W)$: one can show that over algebraically 
closed fields of characteristic~$0$ the resulting $c_*$ is independent of 
the choices and satisfies the Deligne-Grothendieck axioms mentioned 
above (\cite{MR2209219, MR2282409}).

\subsection{A recursive formula for CSM classes of Schubert cells} 
We will now apply identity~\eqref{E:csmdefSNC} to calculate the
Chern-Schwartz-MacPherson 
class of a Schubert cell $X(w)^\circ \subseteq G/B$. 
This class may be viewed as an element of $H_*(G/B)$, and in fact
of $H_*(X(w))$, and hence 
it can be written as an integer linear combination of classes $[X(v)]$
for $v\le w$ in the Bruhat order, as we observed in~\S\ref{ss:FmaSv}.
We will give an algorithm which yields this linear combination.
All the necessary
ingredients were developed in~\S\ref{ss:BS} and we keep the notation
of that section. In particular we recall the fiber diagram
\eqref{E:BSconst}:
\[ 
\xymatrix@C=50pt{
Z \ar[r]^-{\theta_1} \ar[d]_{\pi} & G/B \times_{G/P_{i_k}} G/B 
\ar[r]^-{pr_1}\ar[d]^{pr_2} & G/B \ar[d]^{p_{i_k}} \\ 
Z'\ar[r]^{\theta'} & G/B \ar[r]^{p_{i_k}} & G/P_{i_k}
} 
\] 

Let $s_{i_1}\cdots s_{i_k}$ be any word (reduced or otherwise), and let
$Z:=Z_{i_1,\dots, i_k}$ be the corresponding Bott-Samelson variety.
Recall from~\S\ref{ss:BS} that $G/B \times_{G/P_{i_k}} G/B=\Pbb(\cE)$
for a canonically defined vector rank-$2$ vector bundle $\cE$ on $G/B$; 
thus $Z=\Pbb(\cE')$ is the projectivization of the pull-back 
$\cE'={\theta'}^*(\cE)$. In $H^2(Z_{i_1,\dots, i_k})$ we have the 
tautological class $h_k=c_1(\cO_{\cE'}(1))$, as well as the
pull-backs of $h_j$ from $Z_{i_1,\dots,i_j}$ for $j< k$; we will omit
the pull-back notation.

Let $D_Z$ be the SNC divisor defined in~\S\ref{ss:BS}, and denote by
$Z^\circ=Z_{i_1,\dots, i_k}^\circ$ the complement $Z\smallsetminus D_Z$.
By \eqref{E:csmdefSNC},
\begin{equation}\label{E:CSMZcirc}
\csm(Z^\circ) = 
\frac{ c(T_Z)}{(1+D_1) \cdots (1+D_k)}\cap [Z]
\end{equation}
where $D_1, \dots, D_k$ are the components of $D_Z$. 

\begin{lemma}\label{lem:CSMinBS}
With notation as above, the following holds in $H_*(Z)$:
\[ 
\csm(Z^\circ) =(1+h_k)\cdot \pi^*\big(\csm({Z'}^\circ)\big)
= \prod_{j=1}^k (1+h_j)\cdot [Z]\/. 
\]
\end{lemma}

\begin{proof}
The first formula follows from \eqref{E:CSMZcirc} and 
Proposition~\ref{prop:divclass} (c), and the second formula is
an immediate consequence. 
\end{proof}

Now let $X(w)$ be the Schubert variety determined by $w\in W$, 
that is, the closure of $X(w)^\circ$ in $G/B$, and fix a reduced 
decomposition $s_{i_1}\cdots s_{i_k}$ of $w$ and the corresponding 
Bott-Samelson variety $Z$. As recalled in
Proposition~\ref{prop:resolution}, the composition $pr_1\circ
\theta_1$ gives a proper birational morphism (hence a 
desingularization) $\theta: Z\to X(w)$, restricting to an isomorphism
on $\theta^{-1}(X(w)^\circ) =Z^\circ$.

\begin{lemma}\label{lemma:1rec} 
The CSM class of the Schubert cell $X(w)^\circ$ is given by 
\[ 
\csm(X(w)^\circ) = \theta_*\bigl((1 + h_k)\cdot \pi^*(\csm({Z'}^\circ))
\bigr) \/.
\] 
\end{lemma}

\begin{proof}
By construction, $\theta_*(\one_{Z^\circ})=\one_{X(w)^\circ}$; therefore, the 
functoriality of CSM classes \eqref{E:funCSM} implies $\csm(X(w)^\circ) 
=\theta_* \big(\csm(Z^\circ)\big)$, and the stated formula then follows from 
Lemma~\ref{lem:CSMinBS}.
\end{proof}

Lemma~\ref{lemma:1rec} motivates the study of the quantity 
$\theta_*\big((1+ h_k)\cdot \pi^*(\gamma)\big)$
for $\gamma \in H_*(Z')$. 
The next theorem gives the key formulas needed for explicit calculations,
in terms of the operators introduced in~\S\ref{ss:twoops}.

\begin{theorem}\label{thm:recursion} 
Let $\gamma \in H_*(Z')$. Then the following holds in $H_*(G/B)$: 
\begin{itemize}
\item[(a)] $\theta_*(\pi^*(\gamma)) = 
\partial_{i_k}( \theta'_*(\gamma))$;
\item[(b)] $\theta_*(h_k\cdot \pi^*(\gamma)) = 
-s_{i_k} (\theta'_*(\gamma))$.
\end{itemize}
Therefore, 
\[ 
\theta_*((1+h_k)\cdot \pi^*(\gamma)) = \cT_{i_k}(\theta'_*(\gamma))
\] 
where $\cT_i: H^*(G/B) \to H^*(G/B)$ is the operator given by 
$\cT_i = \partial_i - s_i$.
\end{theorem}

Before proving the theorem, we note that if $\gamma=\csm({Z'}^\circ)$,
then $\theta'_*(\gamma) = \csm(X(w')^\circ)$ where $w' =s_{i_1} \cdots 
s_{i_{k-1}}= w s_{i_k}$. Therefore Theorem~\ref{thm:recursion} gives a 
recursive formula to calculate the CSM classes:

\begin{corol}\label{cor:recursion} 
Let $w \in W$ be a non-identity element and let $s_k$ be a simple reflection 
such that $\ell(ws_k) < \ell(w)$. Then the following recursive identity 
holds: 
\[
\csm(X(w)^\circ) = \cT_k(\csm(X(ws_k)^\circ)) \/, 
\] 
with the initial condition that $\csm(X(\id)^\circ)= \csm(pt) = [pt]$.
\end{corol}

The explicit action of the operator $\cT_k$ on Schubert classes $[X(u)]$
is obtained by combining identities \eqref{E:partial} and
\eqref{E:raction} above. The resulting formula together with other
properties of the operator $\cT_k$ will be presented in~\S\ref{s:propTk}
below.

\begin{proof}[Proof of Theorem~\ref{thm:recursion}] 
Both the left and right squares 
in~\eqref{E:BSconst} are fiber squares, and
$p_{i_k}$ is flat and $\theta=pr_1\theta_1$ is proper, so
\[
\theta_*\pi^*(\gamma) = p_{i_k}^* (p_{i_k})_* \theta'_*(\gamma)
=\partial_{i_k}(\theta'_*(\gamma))
\]
by~\cite[Proposition 1.7]{fulton:it} and the definition of $\partial_{i_k}$
given in~\S\ref{ss:twoops}. This proves (a).

For (b), 
let $\underline\gamma:=\theta'_*(\gamma) \in H^*(G/B)$ and 
\[ 
\tilde{h}_k=c_1(\cO_{\cE}(1)) \in H^2(G/B \times_{G/P_{i_k}} G/B) \/, 
\] 
so that $h_k= \theta_1^*(\tilde{h}_k)$. Then
\begin{align*} 
\theta_*(h_k\cdot \pi^*(\gamma)) 
&= (pr_1)_*(\theta_1)_*(\theta_1^*(\tilde{h}_k)\cdot \pi^*(\gamma)) \\
&= (pr_1)_* (\tilde{h}_k\cdot (\theta_1)_* \pi^*(\gamma)) \\ 
&= (pr_1)_* (\tilde{h}_k\cdot pr_2^* (\theta')_*(\gamma)) \\
&= (pr_1)_* (\tilde{h}_k\cdot pr_2^*(\underline\gamma)) \/.
\end{align*}
In the second equality we used the projection formula, and in the
third we used the fact that the left square in~\eqref{E:BSconst} is a fiber 
square and that $pr_2$ is flat and $\theta'$ is proper. 
Now recall that $G/B$ is the projectivization $\Pbb(E)$ of a vector
bundle $E$ over $G/P_{i_k}$, and $\cE$ is $p_{i_k}^*(E)\otimes \cO_E(1)$
as a bundle over $G/B$. We can compute the tautological subbundle 
$\cO_{\cE}(-1)$ of $pr_2^*(\cE)$, a bundle over $G/B \times_{G/P_{i_k}} G/B$,
by using \cite[Appendix B.5.5]{fulton:it}:
\[
\cO_{\cE}(-1) = \cO_{p_{i_k}^*(E) \otimes \cO_E(1)}(-1) 
= pr_2^* \cO_{p_{i_k}^* (E)}(-1) \otimes pr_2^* \cO_E(1)
=pr_1^*\cO_E(-1) \otimes pr_2^* \cO_E(1)\/.
\] 
Letting $\eta=c_1(\cO_E(1))$, this implies
\[ 
\tilde{h}_k = c_1(pr_1^*\cO_E(1)) + c_1(pr_2^*\cO_E(-1))
= pr_1^*(\eta) - pr_2^*(\eta) \/,
\] 
and the projection formula gives
\begin{align*}
(pr_1)_* (\tilde{h}_k\cdot pr_2^*(\underline\gamma)) 
&= (pr_1)_*\big((pr_1^*(\eta) - pr_2^*(\eta)) 
\cdot pr_2^*(\underline\gamma)\big) \\ 
&=\eta\cdot (pr_1)_*pr_2^*(\underline\gamma) 
-(pr_1)_*pr_2^* (\eta\cdot \underline\gamma) \\
&=\eta\cdot p_{i_k}^* (p_{i_k})_*(\underline\gamma) 
-p_{i_k}^* (p_{i_k})_* (\eta\cdot \underline\gamma)
\end{align*}
where the last equality follows since the second square in
\eqref{E:BSconst} is also a fiber square and $p_{i_k}$ is
both flat and proper. By definition, $\partial_{i_k}=p_{i_k}^* 
(p_{i_k})_*$. Putting all together, we have shown that
\[
\theta_*(h_k\cdot \pi^*\gamma)
=\eta\cdot \partial_{i_k} (\underline\gamma)
-\partial_{i_k}(\eta\cdot \underline\gamma)\/.
\]
Since $\cO_E(1)$ has degree $1$ on the fibers of
$p_{i_k}$, and $p_{i_k}$ has relative dimension~$1$, we have
\[
\partial_{i_k}(\eta)=p_{i_k}^* (p_{i_k})_* (\eta)=[G/B]\/.
\]
We use this and part (b) of Proposition~\ref{prop:BGGop} to get
\begin{align*}
\partial_{i_k}(\eta\cdot \underline\gamma) 
&= \partial_{i_k}(\eta)\cdot \underline\gamma
+\eta\cdot \partial_{i_k}(\underline\gamma)-c_1(\caL_{\alpha_{i_k}})
\cdot\partial_{i_k}(\eta)\cdot \partial_{i_k}(\underline\gamma) \\
&= \underline\gamma
+\eta\cdot \partial_{i_k}(\underline\gamma)-c_1(\caL_{\alpha_{i_k}})
\cdot\partial_{i_k}(\underline\gamma)
\end{align*}
and finally
\[
\theta_*(h_k\cdot \pi^*\gamma)
=-\underline\gamma+c_1(\caL_{\alpha_{i_k}})\cdot
\partial_{i_k}(\underline\gamma) 
=(-\id + c_1(\caL_{\alpha_{i_k}})) (\underline\gamma)
=-s_{i_k} (\underline\gamma)
\]
by \eqref{E:sk}, concluding the proof of (b).
\end{proof}

\subsection{Chern classes of Schubert cells in $G/P$}\label{ss:GmodP}
Fix a parabolic subgroup $P \subset G$ containing the Borel subgroup $B$. 
Let $W_P \subseteq W$ be the subgroup generated by the simple reflections 
in $P$. It is known (see e.g., \cite[\S 1.10]{humphreys:reflection})
that each coset in $W/W_P$ has a unique minimal length
representative; we denote by $W^P$ the set of these representatives. 
If $w \in W$, then one can define a length function $\ell: W/W_P \to \Nbb$
by $\ell(wW_P) := \ell(w')$ where $w' \in W^P$ is in the coset of $w$.

The space $G/P$ is a projective manifold of dimension $\ell (w_0 W_P)$,
where $w_0$ is the longest element in $W$. For each $w \in W^P$ there is 
a Schubert cell $X(wW_P)^\circ := B w P/P$ of dimension $\ell(wW_P)$, 
and the corresponding Schubert variety $X(wW_P):=\overline{B w P/P}$;~see 
e.g., \cite[\S 2.6]{billey.lakshmibai}. The fundamental classes $[X(wW_P)] 
\in H_{2 \ell(wW_P)}(G/P)$ ($w \in W^P$) form a $\Z$-basis for the homology 
$H_*(G/P)$. The natural projection $p: G/B \to G/P$ satisfies $p(X(w)) = 
X(wW_P)$ and the induced map in homology is given by
\begin{equation}\label{E:pf} 
p_* [X(w)] = 
\begin{cases}
[X(wW_P)] & \textrm{if } \ell(w) = \ell(wW_P); \\ 
0 & \textrm{otherwise } \/. 
\end{cases} 
\end{equation} 

\begin{prop}\label{prop:pfcsm}
With notation as above, 
\begin{equation}\label{E:pfcsm}
\csm(X(wW_P)^\circ) = p_*(\csm(X(w)^\circ)) \in H_*(G/P)
\end{equation}
for all $w\in W$. Further, if $u\le w$ and $\ell(u)=\ell(uW_P)$, then the coefficient of $[X(u)]$ in 
$\csm(X(w)^\circ)$ equals the coefficient of $[X(uW_P)]$ in $\csm(X(wW_P)^\circ)$.
\end{prop}

\begin{proof}
The topological Euler characteristic $\chi$ of the fibers of the
restriction of $p$ to $X(w)^\circ$ is constant, hence the push-forward
$p_* (\one_{X(w)^\circ})$ equals $\chi \cdot \one_{X(w
  W_P)^\circ}$. By functoriality of CSM classes \eqref{E:funCSM} this
implies that $p_*(\csm(X(w)^\circ)) = \chi \cdot
\csm(X(wW_P)^\circ)$. Since the coefficient of $[pt]$ in both CSM
classes equals $1$, it follows that $\chi =1$.  The last claim follows
from~\eqref{E:pf}.
\end{proof}

Thus the CSM classes of Schubert cells in $G/P$ are determined by 
the corresponding classes in $G/B$.
For example, the CSM classes of Schubert cells in the ordinary
Grassmannian, determined explicitly in~\cite{aluffi.mihalcea:csm},
can also be computed in principle using the recursive formula obtained
in Corollary~\ref{cor:recursion}; see Example~\ref{ex:Fl4} for a concrete example.
Further, the push-forward formula~\eqref{E:pf} implies that if the positivity
conjecture discussed in~\S\ref{s:posCSM} is true for the CSM classes of 
Schubert cells in $G/B$, then the analogous conjecture must be true for 
CSM classes of Schubert cells in $G/P$ for any parabolic $P$ containing $B$.

%%%

\section{The operators $\cT_k$ and a Weyl group representation 
on $H^*(G/B)$}\label{s:propTk} 
In this section we analyze the operator $\cT_k = \partial_k - s_k: 
H_*(G/B) \mapsto H_*(G/B)$ which gives the recursion for CSM classes
of Schubert cells as proven in~Corollary~\ref{cor:recursion}. 
We start by recording the main algebraic properties of the operators 
$\cT_k$.

\begin{prop}\label{prop:T} The following identities hold:
\begin{itemize}
\item[(a)] $\cT_k^2 = 1$. 
\item[(b)] The operators $\cT_k$ satisfy the braid relations, 
i.e., $(\cT_i \cT_j)^{m_{i,j}} = 1$ where $m_{i,j}$ is 
the order of the element $s_i s_j \in W$.
Also, if $w=s_{i_1}\cdots s_{i_k}$ is a representation of an element 
$w\in W$ as a word in simple reflections, then the operator
$\cT_w:= \cT_{i_1}\cdots \cT_{i_k}$ is independent of the choice
of the word representing $w$.
\item[(c)] For any $u, v \in W$, $\cT_u \cdot \cT_v 
= \cT_{uv}$. 
\end{itemize}
\end{prop} 

\begin{proof} 
First we note that
\[ 
\cT_k^2=(\partial_k - s_k)^2 = s_k^2=1
\]
since $\partial_k^2=0$ and $\partial_k s_k+s_k\partial_k=0$
by Lemma~\ref{lem:commdelksk}. This proves part (a). 
To prove the first part of (b), it suffices to show that the relations
hold after applying the operators to the classes $\csm(X(w)^\circ)$, 
since these form a basis for $H_*(G/B)$. 
These relations follow then immediately from the fact that for all $w\in W$
and all simple reflections $s_k$,
\begin{equation}\label{E:cor3.4upgrade}
\cT_k(\csm(X(w)^\circ)) = \csm(X(ws_k)^\circ)
\end{equation}
as a consequence of Corollary~\ref{cor:recursion} and part (a).
The independence of $\cT_w$ on the specific word for $w$
is also an immediate consequence of~\eqref{E:cor3.4upgrade}.
Finally, (c) follows from the independence of $\cT_w$ on the word
representing~$w$.
\end{proof}

The proposition implies that the operators $\cT_w$ define 
a representation of the Weyl group $W$ on $H_*(G/B)$. 
We record an immediate consequence of the identity 
\eqref{E:cor3.4upgrade} from the proof of Proposition~\ref{prop:T}.

\begin{corol}\label{cor:CSMact} 
Let $u,w$ be two Weyl group elements. Then the identity 
\[ 
\cT_u(\csm(X(w)^\circ)) = \csm(X(w u^{-1})^\circ) 
\]
holds in $H_*(G/B)$. In particular, $\csm(X(w)^\circ)=\cT_{w^{-1}}([pt])$.
\end{corol}

Combining the actions of $\partial_k$ and $s_k$ on Schubert 
classes found in the identities \eqref{E:partial} and \eqref{E:raction} 
from~\S\ref{ss:twoops} we obtain the following explicit formula 
for $\cT_k$:

\begin{prop}\label{prop:Tk} 
\[ 
\cT_k ([X(w)])= 
\begin{cases} 
- [X(w)] & \textrm{if } \ell(ws_k) < \ell(w) \\  
\hphantom{-}[X(ws_k)] + [X(w)] + \sum \langle \alpha_k, 
\beta^\vee \rangle [X(ws_ks_\beta)] & \textrm{if } \ell(ws_k) > \ell(w)
\end{cases} 
\]
where the sum is over all positive roots $\beta\ne \alpha_k$ such that 
$\ell(w) = \ell(ws_k s_\beta)$.
\end{prop}
\begin{example}\label{ex:Fl4}
Using Corollary~\ref{cor:CSMact} and Proposition~\ref{prop:Tk} it is 
straightforward to implement computations
of CSM classes of Schubert cells in symbolic manipulation packages
such as Maple. 
For instance, we obtain that the CSM class for the open cell
in the flag manifold $\Fl(4)$ (in type A) is:
{\small
\begin{multline*}
\csm(X(4321)^\circ) =
 [X(4321)]
+ [X(4312)]
+ [X(4231)]
+ [X(3421)]
+2[X(4213)]
+2[X(4132)] \\
+ [X(3412)]
+2[X(3241)]
+2[X(2431)]
+ [X(4123)]
+5[X(3214)]
+5[X(3142)] \\
+3[X(2413)]
+ [X(2341)]
+5[X(1432)]
+3[X(3124)]
+4[X(2314)]
+6[X(2143)] \\
+4[X(1423)]
+3[X(1342)]
+3[X(2134)]
+4[X(1324)]
+3[X(1243)]
+ [X(1234)]
\end{multline*}
}
\noindent where we use the standard identification of the elements of $W$
with permutations in indexing the $4!$ Schubert classes.

Note that the terms corresponding to the `Grassmannian permutations'
$(a_1 a_2 b_1 b_2)$ with $a_1<a_2$ and $b_1<b_2$ are
\[
[X(3412)]
+3[X(2413)]
+4[X(1423)]
+4[X(2314)]
+4[X(1324)]
+ [X(1234)]
\]
and push-forward as prescribed by identity~\eqref{E:pfcsm} 
in~\S\ref{ss:GmodP} to the CSM class for the open cell in
$G(2,4)$ (cf.~the row corresponding to 
\begin{tikzpicture}
\draw (0,0) --(0,-.5);
\draw (.25,0) --(.25,-.5);
\draw (.5,0) --(.5,-.5);
\draw (0,0) --(.5,0);
\draw (0,-.25) --(.5,-.25);
\draw (0,-.5) --(.5,-.5);
\end{tikzpicture}
in \cite[Example~1.2]{aluffi.mihalcea:csm}).
\qede\end{example}

\begin{remark}\label{rmk:pos} 
Even if $\ell(ws_k)>\ell(w)$, $\cT_k([X(w)])$ is in general 
not a positive combination of Schubert classes. For example, let
$G=\SL_4(\C)$, and let $w = w_0 s_3$, where $w_0$ is the longest
element in $W=S_4$, the symmetric group with $4$ letters. Using again
the standard identification of the elements of $W$ with permutations,
so that $w_0 = (4 3 2 1)$ and $s_3 = (1 2 4 3)$, then $w = (4 3 1 2)$
and 
\[ 
\cT_3([X(4312)]) = [X(4312)] + [X(4321)] - [X(4231)] \/. 
\]

Nevertheless, substantial evidence suggests that the classes 
$\cT_k(\gamma)$, and hence all classes $\cT_w(\gamma)$, {\em are\/}
positive linear combinations of Schubert classes if $\gamma$ is a positive
combination of CSM classes $\csm(X(u)^\circ)$; see~\S\ref{s:posCSM}.
\qede\end{remark}

%%%

\section{Positivity of CSM classes}\label{s:posCSM}

Fix $w \in W$ and consider the CSM class $\csm(X(w)^\circ)$. As we have
shown, if $s_{i_1}\cdots s_{i_k}$ is a reduced decomposition for $w$, then
\begin{equation}\label{E:recap1}
\csm(X(w)^\circ)=\theta_*(\csm(Z^\circ)) = \theta_* \left(
\prod_{j=1}^k (1+h_j)\cdot [Z]
\right)
\end{equation}
where $\theta: Z:=Z_{i_1,\dots,i_k} \to X(w)$ is the Bott-Samelson resolution
(Lemma~\ref{lem:CSMinBS}). We have also shown that
\[
\csm(X(w)^\circ)= \cT_{i_k}\cdots \cT_{i_1}([pt])
\]
(Corollary~\ref{cor:recursion}).
Since $\csm(X(w)^\circ) \in H_*(X(w))$, we have
\begin{equation}\label{E:exp} 
\csm(X(w)^\circ) = \sum_{u \le w} c(u;w) [X(u)] 
\end{equation} 
where $c(u;w)$ are well-defined integers.
In fact, $c(w;w) = 1$ since the map $\theta$ is birational, and $c(\id;w) = 1$ 
since $X(w)^\circ\cong \Abb^{\ell(w)}$ and $\chi(\Abb^{\ell(w)})=1$.

The operator $\cT_k$ does not preserve positivity: 
$\cT_k([X(s_k)^\circ])=-[X(s_k)^\circ]$ by Proposition~\ref{prop:Tk},
and in fact $\cT_k([X(w)^\circ])$ may have negative contributions
from Schubert classes even if $\ell(ws_k)>\ell(w)$ (Remark~\ref{rmk:pos}).
Examples also show that $\csm(Z^\circ)$ is not necessarily a positive 
combination of strata of the normal crossing divisor 
$D_Z:=Z\smallsetminus Z^\circ$ ($Z_{12321}^\circ$ is the smallest such 
example). So one should not expect any positivity properties of the CSM 
class {\it a priori}. Nevertheless, we conjecture that these classes {\em are\/} 
positive:

\begin{conj}\label{conj:pos} 
For all $u \le w$, the coefficient $c(u;w)$ from the expansion \eqref{E:exp} 
is strictly positive. 
\end{conj}

Note that with notation as above, the class of the Schubert {\em variety\/}
$X(w)$ is given by
\begin{equation}\label{E:Svariety}
\csm(X(w))=\sum_{u\le w} \left(\sum_{u\le v\le w} c(u;v)\right) [X(u)]\/;
\end{equation}
indeed, $\one_{X(w)} = \sum_{v\le w} \one_{X(v)^\circ}$.
So Conjecture~\ref{conj:pos} would imply that these classes are also
necessarily effective.

A positivity result analogous to Conjecture~\ref{conj:pos} was conjectured
by the authors in \cite{aluffi.mihalcea:csm} for Schubert cells in the Grassmannian 
$\Gr(p,n)$ of subspaces of dimension $p$ in $\C^n$. This conjecture was proved 
in \cite{aluffi.mihalcea:csm} in the case $p=2$, in \cite{mihalcea:binomial} 
for $p=3$, and several classes of coefficients were proved to be positive 
by B.~Jones \cite{MR2628830} and J.~Stryker \cite{stryker}. 
The full conjecture  has recently been proven by June Huh \cite{huh:csm}. 
By Proposition~\ref{prop:pfcsm}, the CSM classes of Schubert cells in any 
homogeneous space $G/P$ are in fact push-forwards of CSM classes of 
Schubert cells in $G/B$; therefore Conjecture~\ref{conj:pos} would 
simultaneously imply the positivity of all CSM classes of Schubert cells in 
all $G/P$, and in particular it would yield an alternative proof of Huh's 
theorem.

By the same token, Huh's theorem provides some evidence for 
Conjecture~\ref{conj:pos}, since it implies that $c(u;w)>0$ in type A when 
$u$ is a Grassmannian permutation (cf.~Example~\ref{ex:Fl4}).
In fact, Conjecture~\ref{conj:pos} in type A is also supported by explicit 
computations of several thousand cases. At the time of this writing, we 
have verified that the CSM classes of all Schubert cells in $\Fl(n)$ are positive 
for $n\le 7$ and for all words of length $\le 15$ in $\Fl(8)$. 

In the rest of this section we discuss more evidence for 
Conjecture~\ref{conj:pos} in all types. We prove positivity in the following cases:
\begin{itemize}
\item $c(u;w)>0$ if $u<w$ and $\ell(w) - \ell(u) = 1$ (Corollary~\ref{cor:posdiv});
\item $c(u;w)>0$ for all $u\le w$ if $w$ admits a decomposition into distinct 
simple reflections (Corollary~\ref{cor:posdist}).
\end{itemize}
These two results will follow from more general considerations which
seem independently interesting: the first one is an explicit computation
of the codimension~$1$ term in the CSM class of a Schubert cell
(Proposition~\ref{prop:codone}), and the second one highlights one
case in which the operator $\cT_k$ {\em does\/} preserve positivity
(Proposition~\ref{prop:Tkpos}).

\begin{prop}\label{prop:codone} 
Let $\rho=\omega_1 + \cdots + \omega_r$ be the sum of the fundamental
weights, and let $w\in W$. Then 
\[
\csm(X(w)^\circ) = [X(w)]+ c_1(\caL_\rho)\cdot [X(w)]+ \text{lower dimensional terms.}
\]
\end{prop}

\begin{proof} 
Let $w = s_{i_1} \cdots s_{i_k}$ be a reduced decomposition, 
and let $Z:= Z_{i_1, \dots, i_k}$ with SNC divisor $D_Z$, 
as in~\S\ref{ss:BS}. By \cite[Prop. 2.2.2]{brion.kumar:frobenius},
\[ 
K_Z =\cO_Z(- D_Z) \otimes \theta^*(c_1(\caL_{-\rho}))
\] 
and hence $c_1(T_Z)=[D_Z]+\theta^*(c_1(\caL_{\rho}))$.
On the other hand, $c_1(T_Z)=[D_Z]+\sum_{i=1}^k h_i$ by 
Proposition~\ref{prop:divclass}~(c). Therefore
$h_1+\cdots+h_k = \theta^*(c_1(\caL_\rho))$,
and the stated identity follows from~\eqref{E:recap1} and the
projection formula.
\end{proof}

\begin{corol}\label{cor:posdiv}
The coefficient $c(u;w)>0$ if $u<w$ with $\ell(u)=\ell(w)-1$.
\end{corol}

\begin{proof}
Recall that $\langle \omega_i , \alpha_j^\vee\rangle = \delta_{ij}$ 
(the Kronecker symbol) and in particular $\langle \rho,\alpha_j^\vee\rangle>0$ 
for all simple roots $\alpha_j$. Consider $u < w$ such that $\ell(u) = \ell(w) - 1$. 
Then $u = w s_\beta$ for some positive root $\beta \in R^+$ (see e.g. \cite[\S 5.11]{humphreys:reflection}). By the 
Chevalley formula \eqref{E:Chevalley}, the coefficient of $[X(ws_\beta)]$ in 
$c_1(\caL_\rho) \cap [X(w)]$ equals $\langle\rho, \beta ^\vee \rangle > 0$, 
concluding the proof.
\end{proof}

\begin{prop}\label{prop:Tkpos} 
Let $w \in W$ be a Weyl group element, and assume $w$ admits a
decomposition into simple reflections other than $s_k$.
\begin{itemize}
\item[(a)] The homology class $T_k([X(v)])$ is a non-negative linear 
combination of Schubert classes $[X(u)]$ with $u\le vs_k$. In fact,
\[
\cT_k([X(v)]) = [X(vs_k)]+[X(v)]+\sum_{u<vs_k, u\ne v} d_k(u;v)[X(u)]
\]
with $d_k(u;v)\ge 0$ for all $u<vs_k$.
\item[(b)] Assume in addition that $s_k$ commutes with all simple 
reflections in a decomposition of $v$. 
Then $d_k(u;v)=0$ for $u<vs_k$, $u\ne v$, that is:
\[
\cT_{k}([X(v)]) = [X(vs_{k})] + [X(v)] \/. 
\]
\end{itemize}
\end{prop} 

\begin{proof} 
Let $S_v := \{ s_{i_1}, \ldots , s_{i_t} \}$ be the set of reflections 
appearing in a reduced decomposition of $v$; this set is 
independent of the choice of reduced decomposition, since
it is preserved by the braid relations in $W$ (see e.g., 
\cite[\S 5.1]{humphreys:reflection}). Since every decomposition of
$v$ into simple reflections can be reduced to a reduced decomposition,
the hypothesis of the proposition implies that $s_k\not\in S_v$ in 
part~(a), and that further $s_k$ commutes with all $s_{i_j}\in S$ 
in part~(b).

Since $vs_k > v$, by Proposition~\ref{prop:Tk} we have
\begin{equation}\label{E:beta}
\cT_{k}([X(v)]) = [X(vs_{k})] + [X(v)] + \sum \langle \alpha_{k}, 
\beta^\vee \rangle [X(vs_{k}s_\beta)]
\end{equation}
where the sum is over all positive roots $\beta\ne \alpha_{k}$ such 
that $\ell(v) = \ell(vs_{k} s_\beta)$. We have to prove that, under the 
hypothesis of the proposition, 
$\langle \alpha_{k}, \beta^\vee \rangle\ge 0$ for all $\beta$ in the 
range of summation. In fact, all the $\beta$ in this range satisfy 
$vs_k s_\beta< vs_k$, and we will verify that $\langle \alpha_{k}, 
\beta^\vee \rangle\ge 0$ for all such reflections $\beta$.
By \cite[\S 5.7]{humphreys:reflection} the condition $vs_k > v$ 
implies that $v(\alpha_k) >0$, and $vs_{k} s_\beta < v s_k$ 
implies that $v s_k (\beta) < 0$, i.e., 
\begin{equation}\label{E:ineq} 
v ( \beta - \langle \beta, \alpha_k^\vee \rangle \alpha_k) 
= v(\beta) -  \langle \beta, \alpha_k^\vee \rangle v(\alpha_k) <0 \/. 
\end{equation} 
If $v(\beta) >0$, then we are done, because $v(\alpha_k) >0$. 
So we assume $v(\beta) <0$, which is equivalent to $vs_\beta < v$,
and it follows that $vs_\beta$ admits a reduced expression only using 
reflections in $S_v$. We deduce that $s_k$ does not appear in a 
reduced expression for $s_\beta$, and hence that the simple root 
$\alpha_k$ does not appear in the support of the positive root $\beta$. 
Since $s_{\alpha_k}\not\in S_v$, it follows that $\alpha_k$ does not 
appear in the support of $v(\beta)$; and $\alpha_k$ appears with
coefficient $+1$ in $v(\alpha_k)$. Then \eqref{E:ineq} forces 
$\langle \beta, \alpha_k^\vee \rangle \ge 0$ as claimed. 
This proves part (a). 

To prove part (b) we use a similar argument. By \eqref{E:beta},
it suffices to show that $\langle \alpha_k, \beta^\vee \rangle = 0$ 
for all reflections $s_\beta$ such that $v s_{k} s_\beta < vs_k $.
This relation implies that $s_\beta$ has a reduced decomposition 
containing only simple reflections in an expression for $vs_k$. 
But $\beta \neq \alpha_{k}$, and no simple reflections $s_{\alpha_j}$ 
with $\alpha_j$ adjacent to $\alpha_{k}$ in the Dynkin diagram 
for $G$ can appear in the decomposition of $s_\beta$: otherwise such 
reflections would appear in $S_v$, contradicting the commutativity 
hypothesis. This implies that the support of $\beta$ does not contain 
any simple root adjacent to $\alpha_k$, thus  
$\langle \alpha_k, \beta^\vee \rangle = 0$, 
concluding the proof. 
\end{proof} 

\begin{corol}\label{cor:posdist}
Let $w \in W$ be a Weyl group element, and assume $w$ admits a
decomposition into simple reflections other than $s_k$.
If $c(v;w) >0$ for all $v \le w$, then $c(u;w s_k) >0$ for all $u \le ws_k$. 

In particular, if $w\in W$ admits a decomposition into distinct simple
reflections, then $c(u;w)>0$ for all $u\le w$.
\end{corol}

\begin{proof} 
The second statement follows from the first by an immediate induction.
To prove the first statement, note that if $s_k$ does not appear in a
decomposition for $w$, then it does not appear in a reduced
decomposition for $w$, and hence it does not appear in a decomposition
for $v$. Thus the hypothesis of Proposition~\ref{prop:Tkpos} applies
to all $v\le w$.  By Corollary~\ref{cor:recursion} we have
\begin{align*}
\csm(X(ws_k)^\circ) &= \cT_k(\csm(w)^\circ) 
= \cT_k\sum_{v\le w} c(v;w)[X(v)] \\
&= \sum_{v\le w} c(v;w) \left(
[X(vs_k)]+[X(v)]+\sum_{u'<vs_k, u'\ne v} d_k(u';v)[X(u')]
\right)
\end{align*}
with $c(v;w)>0$ by hypothesis and $d_k(u';v)\ge 0$ 
by~Proposition~\ref{prop:Tkpos}.
The statement is immediate from this expression, since $u\le ws_k$
implies that $u$ has a reduced expression which is a subexpression 
of one for $w s_k$, thus either $u=v \le w$ or $u= v s_k$ with $v \le w$. 
\end{proof}

A particular case of Corollary~\ref{cor:posdist} is particularly vivid: 
if $w=s_{i_1}\cdots s_{i_k}$ is a reduced decomposition and the simple 
reflections $s_{i_1}, \dots, s_{i_k}$ {\em commute\/} with one another, 
then an induction argument based on Proposition~\ref{prop:Tkpos}(b) 
implies that
\begin{equation}\label{E:comm}
\csm(X(w)^\circ) = \sum_{u\le w} [X(u)]\/.
\end{equation}

Notice that if $w$ satisfies this condition, then so does every $v$ preceding
it in the Bruhat order. Then \eqref{E:Svariety} and~\eqref{E:comm} 
give the CSM class of the Schubert {\em variety}
\[
\csm(X(w)) =\sum_{u\le w} 2^{\ell(w)-\ell(u)} [X(u)]\/,
\]
for every $w\in W$ decomposing into commuting simple reflections.

%%%

\section{Equivariant Chern-Schwartz-MacPherson classes of Schubert cells}\label{s:equiv}
In this section we extend our calculation of CSM classes to the $T$-equivariant 
situation. We will show that the same difference $\partial_k - s_k$ defines an operator 
$\cT^T_k$ on equivariant homology $H^T_*(G/B)$, sending an (equivariant) CSM class
$\csm^T(X(w)^\circ)$ to the class $\csm^T(X(ws_k)^\circ)$. The proof of 
Lemma~\ref{lem:commdelksk} extends to the equivariant setting, and shows that 
$(\partial_k - s_k)^2 = \id$. In particular, the operators $\cT_k^T$ give a representation 
of $W$ on equivariant homology. The proof that $\cT_k^T$ acts as expected on CSM 
classes is essentially identical to the proof in the non-equivariant case; 
we only need to verify that no `equivariant corrections' are introduced
in the recursion formula from Corollary~\ref{cor:recursion}.

\subsection{Equivariant CSM classes} 
Recall that $T \subset B$ is the maximal torus in the Borel subgroup
$B$. If $X$ is a variety with a $T$-action,
then the equivariant
cohomology $H^*_T(X)$ is the ordinary cohomology of the Borel mixing
space $X_T:= (ET \times X)/T$, where $ET$ is the universal $T$-bundle
and $T$ acts by $t \cdot (e,x) = (e t^{-1}, t x)$. It is an algebra
over $H^*_T(pt)$, a polynomial ring $\Z[t_1, \ldots , t_r]$, where
$t_i, \ldots , t_r$ form generators for the weight lattice of $T$. We
address the reader to \cite{knutson:noncomplex} or \cite{ohmoto:eqcsm} for
basic facts on equivariant cohomology. Since $X$ is smooth, we 
can and will identify the equivariant homology $H_*^T(X)$ with the 
equivariant cohomology $H^*_T(X)$. Every closed subvariety $Y\subseteq X$
that is invariant under the $T$ action determines an equivariant fundamental
class $[Y]_T\in H_*^T(X)$.

Ohmoto defines the group of equivariant constructible functions $\cF^T(X)$ 
(for tori and for more general groups) in \cite[\S2]{ohmoto:eqcsm}.
We recall the main properties that we need: 
\begin{enumerate} 
\item If $W \subset X$ is a constructible set which is invariant under
  the $T$-action, its characteristic function $\one_W$ is an
  element of $\cF^T(X)$. 
(The group $\cF^T(X)$ also contains other elements, but
this will be immaterial for us.)
\item Every proper $T$-equivariant morphism $f: Y \to X$ of algebraic 
varieties induces a homomorphism $f_*^T: \cF^T(X)  \to \cF^T(Y)$. The 
restriction of $f_*^T$ to characteristic functions of constructible $T$-invariant sets 
coincides with the ordinary push-forward $f_*$ of
  constructible functions. See \cite[\S 2.6]{ohmoto:eqcsm}.
\end{enumerate}

Ohmoto
proves \cite[Theorem 1.1]{ohmoto:eqcsm}
that there is an equivariant version of MacPherson
transformation $c_*^T: \cF^T(X) \to H_*^T(X)$ that
satisfies $c_*^T( \one_Y) = c^T(T_Y) \cap [Y]_T$ if $Y$ is a projective,
non-singular variety, and that is functorial with respect to proper
push-forwards. The last statement means that 
for all proper $T$-equivariant morphisms $Y\to X$ the
following diagram commutes:
\[ 
\xymatrix@C=40pt{ 
\cF^T(Y) \ar[r]^{c_*^T} \ar[d]^{f_*^T} & H_*^T(Y) \ar[d]^{f_*^T} \\ 
\cF^T(X) \ar[r]^{c_*^T} & H_*^T(X) 
}
\] 

We denote by $\csm^T(X(w)^\circ):= c_*^T(\one_{X(w)^\circ})$ the
equivariant CSM class of the Schubert cell corresponding to an element $w\in W$.

Finally, we note that both the BGG operator $\partial_i$ and the right
Weyl group action $s_i$ are induced by morphisms which commute with
the $T$-action. It follows that they both determine $H^*_T(pt)$-algebra 
endomorphisms of $H^*_T(G/B)$, for which we will use the
same notation. Further, Proposition~\ref{prop:BGGop} 
and formula~\eqref{E:sk} extend to the equivariant
setting after replacing the Chern class $c_1(\caL_{\alpha_k})$
by its equivariant version $c_1^T(\caL_{\alpha_k})$. We refer
to \cite[\S3]{knutson:noncomplex} for details. The analogue of
identity~\eqref{E:raction} can be found in 
\cite[\S4, Corollary]{knutson:noncomplex};
the equivariant version includes additional terms.

\subsection{Equivariant CSM classes via the operator 
$\cT^T_k= \partial_k - s_k$} 
In this section we give the proof of the equivariant version of
Theorem~\ref{T:mainintro}.

Recall the diagram \eqref{E:BSconst} from~\S\ref{ss:BS}: 
\begin{equation}\label{E:BSconst2} 
\xymatrix{
Z \ar[rr]^{\theta_1} \ar[d]^{\pi}& & G/B \times_{G/P_{i_k}}
G/B \ar[rr]^{pr_1}\ar[d]^{pr_2} && G/B \ar[d]^{p_{i_k}} \\ 
Z'\ar[rr]^{\theta'} && G/B \ar[rr]^{p_{i_k}} && G/P_{i_k}
} 
\end{equation} 
Recall that $Z'$ is the
Bott-Samelson variety for a Weyl group element $w' \in W$ and that $Z$
is the Bott-Samelson variety corresponding to $w's_{i_k}$, where
$\ell(w's_{i_k}) > \ell(w')$.
 
Recall also that $G/B$ is a $\Pbb^1$-bundle $\Pbb(E)$ for some rank-$2$, 
equivariant
vector bundle over $G/P_{i_k}$, and that we defined
$\cE:= E \otimes \cO_E(1)$, a vector bundle over $G/B$. This bundle
fits into the exact sequence~\eqref{E:seqZ} of equivariant vector bundles
\[
\xymatrix{ 
0 \ar[r] & \cO \ar[r] & \cE \ar[r] & \cQ \ar[r]
& 0 \/,
}
\] 
where the action on $\cO$ is trivial, and $\cQ=T_{p_{i_k}}$. The inclusion 
$\cO \subset \cE$ determines an equivariant section $\sigma_{i_k}: G/B \to
\Pbb(\cE) = G/B \times_{G/P_{i_k}} G/B$, inducing the 
section $\sigma$ of $\pi$ introduced in~\S\ref{ss:BS}.

\begin{lemma}\label{lemma:diag} 
The image of the section $\sigma_{i_k}$ is the diagonal $\Delta \subset
G/B \times_{G/P_{i_k}} G/B$. 
\end{lemma}

\begin{proof} 
The image of $\sigma_{i_k}$ is $\Pbb(\cO)$, which maps identically
to $G/B$ via both $pr_1$ and $pr_2$.
\end{proof}

\begin{prop}\label{prop:ident} 
\begin{itemize}
\item[(a)] We have an isomorphism of equivariant bundles 
\[
\cO_{\cE}(1) = pr_2^* \cO_{E}(-1) \otimes pr_1^* \cO_{E}(1) \/. 
\]

\item[(b)] The diagonal $\Delta$ is the zero locus of a homogeneous 
section of the bundle 
\[ 
\cO_{G/B \times_{G/P_{i_k}} G/B}(\Delta) =
pr_2^*(\cQ) \otimes \cO_{\cE}(1) \/. 
\]

\item[(c)] Let $[\Delta]_T \in H_*^T(G/B \times_{G/P_{i_k}} G/B)$ be the
equivariant fundamental class determined by $\Delta$. Then 
\[
[\Delta]_T = c_1^T(\cO_{G/B \times_{G/P_{i_k}} G/B}(\Delta)) \cap 
[G/B \times_{G/P_{i_k}} G/B]_T \/. 
\] 
\end{itemize}
\end{prop}

\begin{proof} 
Part (a) was proved within the proof of Theorem~\ref{thm:recursion}. 
The inclusion $\Delta\subset G/B \times_{G/P_{i_k}} G/B$ is given 
by $\Pbb(\cO) \subset \Pbb(\cE)$. Then by
\cite[Appendix B.5.6]{fulton:it} $\Delta$ is the zero locus of a
section of the bundle $\cQ \otimes \cO_\cE(1)$. 
This establishes (b).

Part (c) follows from a general fact: if an equivariant divisor $D$ is 
the zero scheme of a homogeneous section of an equivariant line bundle $\caL$
on $G/B$, then $[D]_T=c_1^T(\caL)\cap [G/B]_T$. This follows from the
analogous non-equivariant statement in the corresponding Borel 
mixing space.
\end{proof}

The next observation is that~\eqref{E:csmdefSNC} extends to the equivariant case.

Let $Z$ be a variety with a $T$-action,
and let $D\subseteq Z$ be a
divisor with simple normal crossings and equivariant components $D_i$. Then
\begin{equation}\label{eq:eqlogD} 
\csm^T(Z \smallsetminus D) = \frac{ c^T(T_Z)}{\prod_j (1+D_i^T)}\cap [Z]_T \/,
\end{equation}
where $D_i^T=c_1^T(\cO(D_i))$ (so that $D_i^T\cap [Z]_T = [D_i]_T$).

This may be proven by the same method used in the proof of 
\cite[Theorem~1]{aluffi:diff}. 

Now let $Z=Z_{i_1,\dots, i_k}$ and let $D=D_Z$ be the SNC
divisor defined in~\S\ref{ss:BS}. Recall
that $D=\pi^{-1}(D_{Z'}) \cup D_k$. 
The following is the equivariant analogue of Lemma~\ref{lemma:1rec}.

\begin{lemma}
The following identity holds in $H^T_*(G/B)$:
\[ 
\csm^T(X(w)^\circ) = \theta_* ((1 + h_k^T) \cdot  \pi^*(\csm^T(Z'^\circ))) 
\] 
where $h_k^T =(\theta')^* c_1^T(\cO_{\cE}(1))$.
\end{lemma}

\begin{proof} 
By Proposition~\ref{prop:ident}(b), 
$\Delta^T = c_1^T(pr_2^*(\cQ) \otimes \cO_{\cE}(1))$. Pulling back by
$\theta'$, we obtain $D_k^T=c_1^T(\pi^*({\theta'}^*\cQ) \otimes \cO_{\cE}(1))$,
and this implies
\[
c^T(T_Z) = \pi^*(c^T(T_{Z'})) (1+h_k^T)(1+D_k^T) \/,
\]
arguing exactly as in the proof of Proposition~\ref{prop:divclass} (c).
The stated identity follows then from~\eqref{eq:eqlogD} by the same argument
proving Lemma~\ref{lemma:1rec} from~\eqref{E:csmdefSNC}.
\end{proof} 

\begin{theorem}\label{thm:equimain} 
Let $\cT^T_k:H^*_T(G/B) \to H^*_T(G/B)$ be the operator 
\[ 
\cT^T_k =(1 + c_1^T(\caL_{-\alpha_k})) 
\partial_k - \id = \partial_k - s_k \/. 
\] 
Then $\cT^T_k(\csm^T(X(w)^\circ)) =\csm^T(X(ws_k)^\circ)$. 
\end{theorem}

\begin{proof} 
The same proof applies as in the non-equivariant case, taking into
account that all maps used are $T$-equivariant and that the statement
of Proposition~\ref{prop:BGGop} extends without changes to the
equivariant setting. 
\end{proof}

The equivariant versions of the identities~\eqref{E:partial} and~\eqref{E:raction} 
yield the following explicit formula for the equivariant operator $\cT^T_k$:

\begin{prop}\label{prop:equivTk} 
\[ 
\mathcal{T}^T_k ([X(w)])= 
\begin{cases} 
- [X(w)] & \textrm{if } \ell(ws_k) < \ell(w) \\  
(1+ w(\alpha_k))[X(ws_k)] + [X(w)] + \sum \langle \alpha_k, 
\beta^\vee \rangle [X(ws_ks_\beta)] & \textrm{if } \ell(ws_k) > \ell(w)
\end{cases} 
\]
where the sum is over all positive roots $\beta\ne \alpha_k$ such that 
$\ell(w) = \ell(ws_k s_\beta)$, and $w(\alpha_k)$ 
denotes the natural $W$ action on roots.
\end{prop}
Notice that $ws_k > w$ implies that $w(\alpha_k) >0$.

\subsection{Equivariant positivity}
Theorem~\ref{thm:equimain} and Proposition~\ref{prop:equivTk} give an effective 
way to compute equivariant Chern-Schwartz-MacPherson classes of Schubert 
cells in $G/B$. Since the equivariant classes $[X(u)]^T$ form a basis of $H_*^T(G/B)$
over $H_*^T(pt)=\Z[t_1,\dots, t_r]$, we have an equivariant analogue
of~\eqref{E:exp}:
\begin{equation}\label{E:equivexp} 
\csm^T(X(w)^\circ) = \sum_{u \le w} c^T(u;w) [X(u)]_T 
\end{equation} 
where $c^T(u;w)$ are well-defined polynomials in the $t_i$'s. 

\begin{remark}
$\bullet$ 
Proposition~\ref{prop:equivTk} implies that $c^T(u;w)$ is a polynomial 
with integer coefficients in the roots $\alpha_i$; for example, $c^T(u;w)$ is a
polynomial in $t_i-t_{i+1}$ in type~A.

$\bullet$ 
The constant term in $c^T(u;w)$ equals the non-equivariant coefficient $c(u;w)$.

$\bullet$
By general considerations (see e.g.,~\cite[\S4.1]{ohmoto:eqcsm}), 
$\csm^T(X(w)^\circ)$ has no nonzero components in $H^T_i(G/B)$ for $i<0$. 
This implies that the polynomials $c^T(u;w)$ have degree at most $\dim(X(u))
=\ell(u)$ in the roots $\alpha_i$. In particular $c^T(\id;w)$ is constant, and it follows 
that $c^T(\id;w)=1$ for all $w$.

$\bullet$
The sum $\sum_{u,w\in W} c^T(u;w) [X(w)]^T$ equals the total equivariant
Chern class of the flag manifold $G/B$. Now, we have the identity
\[ 
c^T(T_{G/B}) = \prod_{\alpha \in R^+} (1 + c_1^T(\caL_{\alpha})) \/,
\]
where $\mathcal{L}_\alpha$ is the line bundle defined in \S \ref{ss:FmaSv}.
Indeed, 
$T_{G/B}$ is the homogeneous bundle $G \times^B (Lie(G)/Lie(B))= 
G \times^B Lie(U^-)$ where $U^-$ is the opposite unipotent group; it follows 
that the weights at the $B$-fixed point are the negative roots of $(G,B)$. 
In particular, the localization at $\id. B$ of $c^T(T_{G/B})$ equals
\[
\iota_{\id.B}^* c^T(T_{G/B}) = \prod_{\alpha \in R^+} (1 - \alpha) 
= \prod_{\alpha \in R^+} \iota_{\id.B}^* c^T(\caL_{\alpha}) \/. 
\] 
The stated identity follows by homogeneity. 
Since the polynomial $c^T(w_0;w_0)$ equals the coefficient of $[G/B]_T$ in this class,
we must have
\begin{equation}\label{eq:cw0w0}
c^T(w_0;w_0) =\prod_{\alpha\in R^+} (1+\alpha)
\end{equation}
where the product ranges over all positive roots.

$\bullet$ More generally, let $w= s_{i_1}\ldots s_{i_k}$ be a reduced decomposition. 
It follows from Proposition~\ref{prop:equivTk} that the leading term of 
$\csm^T(X(w)^\circ)=(T_{i_k} \ldots T_{i_1}) [X(id)]_T$ is 
\[
c^T(w; w)= \prod_{t=1}^{k} (1 + s_{i_1} \ldots s_{i_{t-1}}(\alpha_{i_t})) \/, 
\]
with the convention that $s_{i_0} = id$. If $w=w_0$, then the set of roots 
$s_{i_1} \ldots s_{i_{t-1}}(\alpha_{i_t})$ obtained as $t$ varies coincides with the set of 
positive roots, and one recovers \eqref{eq:cw0w0}. 
\qede\end{remark}

The following statement generalizes Conjecture~\ref{conj:pos}.

\begin{conj}\label{conj:equipos}
For all $u \le w$, the coefficient $c^T(u;w)$ is a polynomial with positive coefficients
in the simple roots $\alpha_i$. 
\end{conj}

This statement is supported by explicit computations for low dimensions in type A.

\begin{example}
Let $\Gamma_T$ be the matrix whose $(u,w)$ entry is $c^T(u;w)$. 
For $Fl(3)$, listing the elements of $S_3$ in the order $(123), (132), (213), (231), (312), (321)$,
we have
{\tiny
\[\Gamma_T=
\begin{pmatrix}
1 & 1 & 1 & 1 & 1 & 1 \\
0 & 1+\alpha_2 & 0 & 2+\alpha_1+\alpha_2 & 1+\alpha_2 & 2+\alpha_1+\alpha_2 \\
0 & 0 & 1+\alpha_1 & 1+\alpha_1 & 2+\alpha_1+\alpha_2 & 2+\alpha_1+\alpha_2 \\
0 & 0 & 0 & (1+\alpha_1)(1+\alpha_1+\alpha_2) & 0 & (1+\alpha_1)(1+\alpha_1+\alpha_2) \\
0 & 0 & 0 & 0 & (1+\alpha_2)(1+\alpha_1+\alpha_2) & (1+\alpha_2)(1+\alpha_1+\alpha_2) \\
0 & 0 & 0 & 0 & 0 & (1+\alpha_1)(1+\alpha_2)(1+\alpha_1+\alpha_2) \\
\end{pmatrix}
\]
}
verifying Conjecture~\ref{conj:equipos} in this case.
\qede\end{example}

%%%

\bibliographystyle{halpha}
\bibliography{csm_flagsbib}

%%%

\end{document}